\documentclass{article}

\usepackage{amssymb}
\usepackage{amsthm}
\usepackage{amsmath,amssymb,amsopn,amsfonts,mathrsfs,amsbsy,amscd}
\usepackage{longtable}
\usepackage{caption}
\usepackage{multirow}
\usepackage{lineno}
\usepackage{systeme}
\newcommand{\X}{{\mathfrak X}}
\newcommand{\lr}{\longrightarrow}
\newcommand{\we}{\wedge}

\newcommand{\om}{\omega}
\newcommand{\esp}{\quad\mbox{and}\quad}
\newcommand{\G}{{\mathfrak{g}}}
\newcommand{\h}{{\mathfrak{h}}}
\newcommand{\ad}{{\mathrm{ad}}}

\newcommand{\aff}{{\mathrm{aff}}}

\newcommand{\md}{{\mathrm{d}}}
\newcommand{\al}{\alpha}
\newcommand{\alb}{\overline\al}
\newcommand{\omb}{\overline\om}
\newcommand{\hb}{\overline\h}
\newcommand{\xib}{\overline\xi}
\newcommand{\Gb}{\overline\G}
\newcommand{\be}{\beta}
\newcommand{\R}{\mathbb{R}}
\newcommand{\la}{\lambda}

\newtheorem{theo}{Theorem}[section]
\newtheorem{pr}{Proposition}[section]
\newtheorem{Le}{Lemma}[section]
\newtheorem{co}{Corollary}[section]

\newtheorem{exems}{Examples}
\newtheorem{remark}{Remark}
{%
	\stepcounter{equation}
	\begin{equation*}}{%
	\leqno(\arabic{equation})
	\end{equation*}
}

\title{On cosymplectic Lie Algebras}

\author{S. El bourkadi and  M. W. Mansouri\\Universit\'e Ibn Tofail\\ Facult\'e des Sciences. Laboratoire L.A.G.A\\ K\'enitra-Maroc\\e-mail: mansourimohammed.wadia@uit.ac.ma\\
	said.elbourkadi@uit.ac.ma}

\begin{document}
\maketitle

\begin{abstract}
We give some properties of cosymplectic Lie algebras, we show, in particular, that they support a left symmetric product. We also give some constructions of cosymplectic Lie algebras, as well as a classification in three and five-dimensional cosymplectic Lie algebras.
\end{abstract}

key words:
Cosymplectic structures, Left-symmetric product, Double extensions.\\
AMS Subject Class (2010):  53D15, 22E25.

\section{Introduction}

Cosymplectic manifolds were introduced by Libermann In 1958.  She defined it as: An {\it{ almost cosymplectic structure}} on a  manifold $M$ of odd dimension ($2n+1$) is a pair $(\al,\om)$, where $\al$ is a $1$-form and $\om$ is a $2$-form such that 
$\al\we\om^n$ is a volume form on $M$. The structure is
said to be  cosymplectic if $\al$ and $\om$ are closed. Any almost cosymplectic structure $(\al,\om)$ uniquely determines a smooth vector field
$\xi$ on $M$, called the {\it{Reeb vector field}} of almost cosymplectic manifold $(M,\al,\om)$ and it is completely characterized by the following conditions

\begin{equation} \label{1}
\al(\xi)=1\esp \iota_\xi\om=0,
\end{equation}
where $\iota$ denotes
the inner product.  If we consider the vector bundle morphism  $\varPhi : \X(M) \lr\Omega^1(M)$
defined by
\begin{equation} \label{2}
\varPhi(X)=\iota_X\om+\al(X)\al.
\end{equation}
The condition that $\al\we\om^n$ is a volume form  is equivalent to the condition that $\varPhi$ is a vector
bundle isomorphism.  In this case the Reeb vector is given by $\xi= \varPhi^{-1}(\al)$.

For more details on cosymplectic geometry, we refer the reader to the survey article \cite{C-N-Y} and the references therein.

In this paper, we are interested in Lie groups admit a cosymplectic structure which is invariant under left translations (left invariant). 
Let $G$ be a ($2n+1$)-dimensional real Lie group  and $\G$  the corresponding Lie algebra. If $G$ is endowed with a left invariant differential  $1$-form $\al^+$  and  $2$-form $\om^+$  such that $(\al^+,\om^+)$ is a cosymplectic structure, we will say that  $(G,\al^+,\om^+)$ is a  {\it{ cosymplectic  Lie group}} and that
$(\G,\al,\om)$ is a  {\it{ cosymplectic  Lie algebra}}, where $\al=\al^+(e)$ and $\om=\om^+(e)$, with $e$ is the unit element of $G$.  $(\G,\al,\om)$ is a   cosymplectic  Lie algebra this is equivalent to  
\begin{enumerate}
	\item $\al([x,y])=0$, \qquad\qquad $\forall x, y\in\G$.
	\item $\om([x,y],z)+\om([y,z],x)+\om([z,x],y)=0$,  \qquad $\forall x, y, z\in\G$.
	\item $\al\we\om^n\not=0$.
\end{enumerate}
The Reeb vector field is the unique left invariant vector field $\xi^+$ satisfying $\al(\xi)=1$ and $\iota_\xi\om=0$, where $\xi=\xi^+(e)$ is called the Reeb vector of $(\G,\al,\om)$. Note that from 1. a semi-simple Lie algebra (in particular if $[\G,\G]=\G$) cannot support a cosymplectic structure. See \cite{C-P} for the study of cosymplectic Lie algebras and a characterization.

Recall that a finite-dimensional algebra $(\G,.)$  is called {\it{left-symmetric}} if it satisfies the identity
\begin{equation*}
ass(x,y,z)=ass(y,x,z)\qquad \forall x,y,z \in\G,
\end{equation*}
where $ass(x,y,z)$  denotes the associator $ass(x,y,z)=(x.y ).z-x.(y.z)$.  In this case, the commutator $[x,y]= x.y-y.x$
defines a bracket that makes $\G$  a Lie algebra.
Clearly, any {\it{associative algebra}} product (i.e. $ass(x,y,z)=0$, $\forall x,y,z\in\G$) is  a left symmetric product.

A {\it{symplectic Lie algebra}} $(\G,\om)$ is a real Lie algebra with a skew-symmetric
non-degenerate bilinear form $\om$ such that for any $x$, $y$, $z\in\G$,
\begin{equation}\label{3}
\oint\om([x,y],z)=0,
\end{equation}
this is to say, $\om$ is a non-degenerate $2$-cocycle for the scalar cohomology of $\G$, where $\oint$ denotes summation over the cyclic permutation.

It is known that  (see \cite{C} and \cite{M-R}) the product given by
\begin{equation}\label{4}
\omega(x*y,z)=-\omega(y,[x,z]),\qquad\forall x,y,z\in\G,
\end{equation}

induces a left symmetric algebra structure that satisfies $x*y-y*x=[x,y]$ on $\G$, we say that the left symmetric product is {\it{associated with   the symplectic Lie algebra}} $(\G,\om)$.
Geometrically, this is equivalent to the existence in a symplectic Lie group of {\it{left-invariant affine structure}} (a left-invariant linear connection   with zero torsion and zero curvature).

The paper is organized as follows. In Section $2$, we show that any cosymplectic Lie algebra supports a left symmetric product and we give some properties. In Section $3$, we give some procedures to construct cosymplectic Lie algebras. In particular, we suggest two different constructions of cosymplectic double extension. In the last section we will  give some results in low dimension and we also  give a classification in three and five-dimensional cosymplectic Lie algebras.

\textit{Notations}:  Let $\{e_i\}_{1\leq i\leq n}$  a basis of $\G$, we denote by $\{e^i\}_{1\leq i\leq n}$ its
dual basis on $\G^\ast$ and  $e^{ij}$  the 2-form $e^i\wedge e^j\in\wedge^2\G^*$. Set by  $\langle e \rangle:= span\{e\}$ the one-dimensional trivial Lie algebra.

The software Maple 18$^\circledR$ has been used to check all needed calculations.
\section{Left-symmetric product associated with  cosymplectic Lie algebra}

It is known that any symplectic Lie algebra can be equipped with an affine structure, on the other hand, a contact Lie algebra does not necessarily admit a left symmetric product. Our main result is to show that any cosymplectic Lie algebra supports a left symmetric product. 

In the following, $(\G,\al,\om)$ is a ($2n+1$)-dimensional real cosymplectic Lie algebra with the Reeb vector $\xi$. Therefore, we have an isomorphism
\begin{equation}\label{5}
\begin{array}{rcl}
\varPhi : \G& \lr & \G^* \\
x & \longmapsto & \iota_x\om+\al(x)\al. \\
\end{array}
\end{equation}
Throughout the remainder of this paper posing $\h= \ker\al$ and $\om_{\h}=\om_{|\h\times\h}$.
\begin{Le}
	Let $(\G,\al,\om)$ be a cosymplectic Lie algebra with the Reeb vector $\xi$.
	Then $\h$ is an ideal of $\G$ and $ (\h,\om_{\h}) $ is a symplectic Lie algebra.
\end{Le}

\begin{proof} Let $x\in\h$ and $y\in\G$ we have
	\[\al([x,y])=-\md\al(x,y)=0,\]
	then $\h$ is an ideal of $\G$. Now proving that  $ (\h,\om_{\h}) $ is a symplectic Lie algebra, it is clear that $\om_\h$ is 2-cocycle, consider a basis
	$\{\xi,e_1,...,e_{2n}\}$ of $\G$, with $\{e_1,...,e_{2n}\}$ is a basis of $\h$. We have
	\[0\not=\al\we\om^n(\xi,e_1,...,e_{2n})=\om_{\h}^n(e_1,...,e_{2n}).\]
	Then  $ (\h,\om_\h) $ is a symplectic Lie algebra.
\end{proof} 
Denote by $*$ the left-symmetric product associated with the symplectic Lie algebra $(\h,\om_{\h})$ and $ass^*$ its associator (i.e. $ass^*(x,y,z)=(x*y)*z-x*(y*z)$, $\forall x,y,z\in\h$ ).

As in the symplectic framework, the non-degeneration of $\varPhi$ defines a product on $\G$ by
\begin{equation} \label{6}
\varPhi(x.y)(z)=- \varPhi(y)([x,z])\qquad x,y,z\in\G.
\end{equation}
\begin{pr}\label{pr2.1}
	The product defined by $(\ref{6})$, is characterized by 
	\begin{enumerate}
		\item  For $x,\,y\in\h$, we have   
		\[x.y=x*y+\om(x,\ad_{\xi}y)\xi.\]
		\item For $ x\in\G$,  we have 
		\[\xi.x=\ad_\xi x\esp  x.\xi=0.\] 
	\end{enumerate}	
	
\end{pr}
\begin{proof} 
	\begin{enumerate}
		\item For all $x,y\in\h$ and $z\in\G$ the relation $(\ref{6})$ becomes
		\begin{align*}
		\om(x.y,z)+\al(x.y)\al(z)&=-\om(y,[x,z])-\al(y)\al([x,z])\qquad\\
		&=-\om(y,[x,z]).
		\end{align*}
		If $z\in\h$,  we have $\om_{|\h}(x.y,z)=-\om_{|\h}(y,[x,z])$
		and if $z=\xi$,  we have
		\begin{align*}
		\al(x.y)&=-\om(y,[x,\xi])\\
		&=\om(x,[\xi,y]).
		\end{align*}
		This shows 1.
		\item For $x=\xi$ and $y\in\h$, the relation $(\ref{6})$ becomes
		\begin{align*}
		\om(\xi.y,z)+\al(\xi.y)\al(z)&=-\om(y,[\xi,z]).
		\end{align*}
		On the one hand, for $z=\xi$ we obtain $\al(\xi.y)=0,$ then  $\xi.y\in\h$, on the other hand for $z\in\h$, we have 
		\begin{align*}
		\om_{\h}(\xi.y,z)&=-\om_{\h}(y,[\xi,z])\\
		&=\om_{\h}([\xi,y],z),
		\end{align*}
		hence, $\xi.x=\ad_\xi x$, $\forall x\in\G$.
		
		Finally, for $y=\xi$ we also have
		\begin{align*}
		\om(x.\xi,z)+\al(x.\xi)\al(z)&=0,
		\end{align*} 
		then $x.\xi=0$,  $\forall x\in\G$.
	\end{enumerate}
	
\end{proof}
The following lemma shows that $\ad_{\xi}$ is a derivation relatively to the left-symmetric product associated with the symplectic Lie algebra $(\h,\om_{\h})$.
\begin{Le}\label{l2}
	For all $x$, $y\in\h$ we have
	\begin{equation*}
	\ad_{\xi}(x*y)=\ad_{\xi}x*y+x*\ad_{\xi}y.
	\end{equation*}
\end{Le}
\begin{proof} 	
	First, we note that $\ad_\xi x\in\h$ for all $\in\h$. Now for all $x$, $y$,$z\in\h$ we have	
	(by using relations $(\ref{3})$, $(\ref{4})$ and Jacobi identity):
	\begin{align*}
	\om_{\h}(\ad_{\xi}(x*y)-\ad_{\xi}x*y,z)
	&=\om([\xi,x*y],z)+\om(y,[[\xi,x],z])\\
	&=-\om([z,\xi],x*y)+\om(y,[[\xi,x],z])\\
	&=\om([x,[z,\xi]],y)-\om([[\xi,x],z],y)\\
	&=\om([[x,z],\xi],y)\\
	&=-\om([\xi,y],[x,z])\\
	&=\om_{\h}(x*\ad_{\xi}y,z).
	\end{align*}
\end{proof}
The following theorem shows in particular, that a cosymplectic Lie algebra supports a left-symmetric product structure. 
\begin{theo} 
	The product defined by 
	\begin{equation} 
	\varPhi(x.y)(z)=- \varPhi(y)([x,z])\qquad x,y,z\in\G,
	\end{equation}
	is a left-symmetric product in $\G$.
\end{theo}
\begin{proof}
	On the one hand, for all $x$, $y$ and $z\in\h$, we have
	\begin{align*}	
	ass(x,y,z)&=(x.y).z-x.(y.z)\\
	&=(x*y+\om(x,[\xi,y])\xi).z-x.(y*z+\om(y,[\xi,z])\xi)\\
	&=(x*y).z+\om(x,[\xi,y])\xi.z-x.(y*z)\\
	&=(x*y)*z+\om(x*y,[\xi,z])\xi+\om(x,[\xi,y])\xi.z-x*(y*z)-\om(x,[\xi,y*z])\xi\\
	&=(x*y)*z-x*(y*z)+\om(x,[\xi,y])\xi.z+A(x,y)\xi\\
	&=ass^{*}(x,y,z)+\om(x,[\xi,y])\xi.z+A(x,y)\xi.
	\end{align*}
	With $A(x,y)=\om(x*y,[\xi,z])-\om(x,[\xi,y*z])$. It is clear that $ass^{*}(x,y,z)=ass^{*}(y,x,z)$,  using $(\ref{3})$ and the fact that $\om(\xi,.)=0$, we get
	\[\om(x,[\xi,y])=\om(y,[\xi,x]).\]
	
	We also have
	\begin{align*}	
	A(x,y)-A(y,x)&=\om([x,y],[\xi,z])-\om(x,[\xi,y*z])+\om(y,[\xi,x*z])\\
	&=-\om(z,[[x,y],\xi])+\om(y*z,[x,\xi])-\om(x*z,[y,\xi])\\
	&=\om(z,[\xi,[x,y]])+\om(z,[y,[\xi,x]])+\om(z,[x,[y,\xi]])\\
	&=0.
	\end{align*}
	It follows that, $ass(x,y,z)=ass(y,x,z)$ for all $x$,$y$,$z\in\h$.
	
	On the other hand,  for all $x$ and $y\in\h$, we have
	\begin{align*}	
	ass(\xi,x,y)-ass(x,\xi,y)&=(\xi.x).y-\xi.(x.y)-(x.\xi).y+x.(\xi.y)\\
	&=(\ad_{\xi}x)*y+\om([\xi,x],[\xi,y])\xi-\ad_{\xi}(x.y)+x*\ad_{\xi}y+\om(x,[\xi,[\xi,y]])\xi.
	\end{align*}
	Then	$ass(\xi,x,y)-ass(x,\xi,y)=0$ is equivalent to
	
	\[
	\left\{
	\begin{array}{l}
	\ad_{\xi}(x*y)=\ad_{\xi}x*y+x*\ad_{\xi}y\\
	\om([\xi,x],[\xi,y])=-\om(x,[\xi,[\xi,y]]).
	\end{array}
	\right.
	\]
	The first equation is ensured by Lemma $\ref{l2}$  while the second follows from the closure of $\om$  and from the fact that $\om(\xi,.)=0$. 
\end{proof}
It is known that (see, for example, \cite{M} Proposition 1-1) in a Lie group an affine connection is bi-invariant if and only if its associated left symmetric product is associative. We get the following result.
\begin{co}\label{co 2.1}
	Let $(G,\al^+,\om^+)$ be a cosymplectic  Lie group with the Reeb vector $\xi^+$, the affine	structure associated  is bi-invariant if and only if the following conditions are satisfied 
	\begin{enumerate}
		\item $ass^*(x,y,z)=\om(\ad_\xi y,x)\,\ad_\xi z$.
		\item $x*\ad_\xi y=0$.
		\item $(\ad_\xi x)*y=(\ad_\xi y)*x.$
		\item $\ad^2_\xi x=0$.
	\end{enumerate}	
	for all $x,y,z\in\h$. 
\end{co}
\begin{remark}
	Note that if $\ad_\xi x=0 $, then the relation $1.$ becomes $ass^*(x,y,z)=0\quad \forall x,y,z\in\h $, this means that the left symmetric product $*$ associated with the symplectic Lie algebra $ (\h,\om_{\h}) $ is associative. Which equivalent to say that $ (\h,\om_{\h}) $ is a {\it{Novikov Lie algebra}} ( see \cite{A-M}  for more details).
\end{remark}

\section{Cosymplectic double extensions}
We suggest two different constructions of  cosymplectic double extension. In the first construction, we combine two characterizations of cosymplectic Lie algebras. The second way of construction is inspired  from the notion of  symplectic double  extension.

\textbf{Double extensions of Lie algebras}: Let $(\Gb,\overline{[\,,\,]})$ be a Lie algebra and let $\theta\in Z^2(\Gb)$ be a $2$-cocycle. We will denote by $(\Gb_{(\theta,e)},\overline{[\,,\,]}_\theta)$
the central extension of  $\Gb$  by the 2-cocycle $\theta$, i.e.,
\begin{eqnarray*}
	\Gb_{(\theta,e)}=(\Gb\oplus \langle e\rangle, [.,.]_\theta)\qquad with\qquad
	[x,y]_\theta &=& \overline{[x,y]}+\theta(x,y)e\quad x,y\in\Gb. 
\end{eqnarray*}

Let \[\G=\langle d\rangle\oplus\Gb\oplus \langle e\rangle\] 
be a direct sum of $\Gb$ with  one-dimensional vector spaces $\langle e\rangle$ and $\langle d\rangle$ .
Define an alternating bilinear product  $[.,.] : \G\times\G\lr\G$ by
\begin{align*}
[x,y] &= \overline{[x,y]}+\theta(x,y)e,\quad\qquad x, y\in\Gb, \\
{[d,x]} &= \varphi(x)+\lambda(x)e, \qquad\qquad x\in\Gb,\\
{[d,e]} &=v+te,
\end{align*}
with, $D=(\varphi,\lambda,v,t)\in End(\Gb)\times\Gb^*\times\Gb\times\R$.

Note by
\begin{align*}
\partial\varphi(x,y)&=\varphi(\overline{[x,y]})-\overline{[\varphi(x), y]}-\overline{[x,\varphi(y)]}\\ 	\theta_\varphi(x,y)&=\theta(\varphi(x),y)+\theta(x,\varphi(y)).
\end{align*}
We have the following result.
\begin{pr}\label{pr 3.1}
	The alternating product $[\,,\,]$  above, defines a Lie algebra  $(\G,[\,,\,])$ if and only if
	\begin{enumerate}
		\item $\partial\varphi=\theta v.$ 
		\item $t\theta-\theta_\varphi=\md\lambda.$
		\item $	v\in Z(\Gb) \cap\ker(\theta)$,
	\end{enumerate}
	where $\ker(\theta)=\{x\in\Gb\; |\;\theta(x,y) = 0,\quad \forall y \in\Gb\}$.
\end{pr}

\begin{proof}
	On the one hand, for $x$,	$y\in\Gb$ we have
	\begin{align*}
	\oint[[d,x],y] &=[[d,x],y]+[[x,y],d]+[[y,d],x] \\
	&=[\varphi(x)+\la(x)e,y]+[\overline{[x,y]}+\theta(x,y)e,d]-[\varphi(y)+\la(y)e,x] \\
	&=\overline{[\varphi(x),y]}+\theta(\varphi(x),y)e- \varphi(\overline{[x,y]})-\la(\overline{[x,y]})e-\theta(x,y)v\\
	&-t\theta(x,y)e-\overline{[\varphi(y),x]}-\theta(\varphi(y),x)e\\
	&=\big(\partial\varphi(x,y)-\theta(x,y)v\big)+\big(\theta_\varphi(x,y)-t\theta(x,y)-\la(\overline{[x,y]})\big)e,
	\end{align*}
	hence $1.$ and $2.$.
	
	On the other hand, we have for $x\in\Gb$
	\begin{align*}
	\oint[[x,e],d]
	&=[[x,e],d]+[[e,d],x]+[[d,x],e]\\
	&=0+[v+te,x]+0\\
	&=\overline{[v,x]}+\theta(v,x)e, 
	\end{align*}
	hence $v\in Z(\Gb)\cap\ker(\theta)$. The other possible Jacobi identities are immediately checked.
\end{proof}
If the conditions of Proposition $\ref{pr 3.1}$ are hold, then the Lie algebra $\Gb(D,\theta)=(\G,[\,,\,])$ is called a {\it{double extension}}   of $(\Gb,\overline{[\,,\,]})$ by $(D,\theta)$, with $D=(\varphi,\lambda,v,t)\in End(\Gb)\times\Gb^*\times\Gb\times\R$.
\begin{remark}
	\begin{enumerate}
		\item The Lie algebra $\Gb(D,\theta)$ is a semi-direct product $\langle d\rangle\ltimes(\Gb\oplus \langle e\rangle)$ of the abelian Lie algebra $\langle d\rangle$ with central extension $\Gb\oplus \langle e\rangle$ relatively to the derivation $D\in Der( \Gb\oplus \langle e\rangle)$ given by
		\begin{align*}
		D(x) &= \varphi(x)+\lambda(x)e, \qquad x\in\Gb\\
		D(e) &= v+te.
		\end{align*}	
		\item  $\varphi$  is a derivation of the Lie algebra $\Gb$ if and only if $v=0$ or $\theta=0$. In particular, this holds when either $Z(\Gb)=\{0\}$ or $\ker(\theta)=\{0\}$.
	\end{enumerate}
\end{remark}

\subsection{The  first construction}

We start with two characterizations of symplectic Lie algebras. The first was given in \cite{L-S} characterizes  cosymplectic manifolds. We state here its analogue for cosymplectic Lie algebras.
\begin{pr}\cite{L-S}\label{pr 3.2}
	Let $\Gb$ be a Lie algebra and $\alb$, $\omb$ two  forms on $\Gb$ with degrees $1$ and $2$ respectively. Consider on the direct sum $\G=\Gb\oplus\langle e \rangle $ the 2-form $\om=\omb+\alb\we e^*$. Then $(\Gb,\alb,\omb)$  is a cosymplectic Lie algebra if and only if $(\G,\om)$ is a symplectic Lie algebra.
\end{pr}
\begin{remark}
	The Lie algebra $\G=\Gb\oplus\langle e \rangle $ is a central extension of $\Gb$ by $\theta=0$. 
\end{remark}
As for the second characterization (see \cite{B-F-M} and \cite{C-P}).

\begin{pr}\label{pr 3.3}
	There exist a one-to-one correspondence between ($2n+1$)-dimensional
	cosymplectic Lie algebras $(\G,\al,\om)$, and ($2n$)-dimensional symplectic Lie algebras $(\h=\ker \al,\om_{\h})$  together with a derivation $D\in Der(\h)$, such that $D$ satisfied
	\begin{equation}\label{ist}
	\om_h(Dx,y)=-\om_{\h}(x,Dy),\qquad x,y\in\h.
	\end{equation}
\end{pr}
An  endomorphism  satisfied $(\ref{ist})$ is called an infinitesimal symplectic transformation (for short, an i.s.t.).

By combining these two characterizations, we propose the following constructions of cosymplectic Lie algebras. 

Let $(\Gb,\alb,\omb)$ be a cosymplectic Lie algebra with the Reeb vector $\xib$, recall that we can always write $ \Gb$ as  $\Gb=\hb\oplus\langle\overline\xi\rangle$ with $\hb=\ker(\alb)$.
Let  $(\Gb\oplus\langle e \rangle, {\om}=\omb+\al\we e^*)$ be a symplectic Lie algebra associated with $(\Gb,\alb,\omb)$  (see Proposition $\ref{pr 3.2}$). A derivation  $D \in Der(\Gb\oplus\langle e  \rangle)$ consists of a $4$-tuple $(\varphi,\lambda,v,t)\in End(\Gb)\times\Gb^*\times\Gb\times\R$ such that
\begin{align*}
D(x) &= \varphi(x)+\lambda(x)e, \qquad x\in\Gb,\\
D(e) &= v+te.
\end{align*}
\begin{Le}\label{L3.1}
	The derivation $D \in Der(\G)$  is an i.s.t if and only if the fillowing four conditions are satisfied
	\begin{enumerate}
		\item[(i)] $\varphi$ is an i.s.t in $\hb$.
		\item[(ii)]  $\la(x)= \omb(x,\varphi(\xib))$,\qquad for all $x\in\hb$.
		\item[(iii)]	$\omb(v,x)=\alb\circ\varphi(x)$,\qquad for all $x\in\hb$.
		\item[(iv)] $t=-\alb\circ\varphi(\xib)$.
	\end{enumerate}
\end{Le}
\begin{proof} 
	Indeed. For $x$, $y\in\Gb$,  we have
	\begin{align}\label{8}
	\om(Dx,y)+\om(x,Dy)&=  \omb(\varphi(x),y)+\omb(x,\varphi(y))+\la(x)\om(e,y)+\la(y)\om(x,e).
	\end{align}
	If $x\in\hb$ and $y\in\hb$. Using $(\ref{8})$, $D$ is an i.s.t yeilds to $\omb(\varphi(x),y)=-\omb(x,\varphi(y))$.
	
	If  $x\in\hb$ and $y=\xi$. Once again, by applying $(\ref{8})$, the fact that $D$ is an i.s.t gives $\la(x)= \omb(x,\varphi(\xi))$.
	
	Now for $x\in\hb$ and $y=e$, we have
	\begin{align*}
	\om(Dx,e)+\om(x,De)&= \om(\varphi(x)+\la(x)e,e)+\om(x,v+te)\\
	&= \om(\varphi(x),e)+\omb(x,v),
	\end{align*}
	since $D$ is an i.s.t, it follows that $\omb(x,v)=\alb\circ\varphi(x)$.
	
	Finally, for $x=\xib$ and $y=e$, we have
	\begin{align*}
	\om(D\xib,e)+\om(\xib,De) &=\om(\varphi(\xib)+\la(\xib)e,e)+\om(\xib,v+te)\\ 
	&=\om(\varphi(\xib),e)+t,
	\end{align*}
	so $t=-\alb\circ\varphi(\xib)$. This completes the proof.
\end{proof}
Let $D=(\varphi,\la,v,t)$ satisfying  Lemma \ref{L3.1} conditions.

Denote by $*$ the left-symmetric product associated with the symplectic Lie algebra $(\hb,\omb)$ and let $\overline{R}_x$ be the right multiplication by an element $x$, that is $\overline{R}_xy=y*x$, we obtain the following result.
\begin{theo}
	Let $(\Gb,\alb,\omb)$ be a cosymplectic Lie algebra with the Reeb vector $\xib$. Let  $\G=\langle d \rangle\oplus\Gb\oplus\langle e \rangle$ be a double extension   of $\Gb$, by $(D,\theta=0)$, with $D=(\varphi,\lambda,v,t)$ satisfies the hypotheses of Lemma \ref{L3.1}. Then 
	\[{\om}= \omb+\alb\we e^*,\qquad
	{\al}= d^*,\]
	defines a cosymplectic structure in $\G$ if and only if
	\begin{enumerate}
		\item $\varphi\in Der(\Gb)$,
		\item    $v\in Z(\Gb),$
		\item $\overline{R}_{\varphi(\xib)}=0$ and $[\varphi(\xib),\xib]=0$.
		
	\end{enumerate}
	Also, the Reeb vector is $d$.
\end{theo}
\begin{proof} With $\theta=0$,  Proposition \ref{pr 3.1} gives $\varphi\in Der(\Gb)$, $v\in Z(\Gb)$ and $\md\la=0$.
	\begin{eqnarray*}	
		\md\la=0	&\Leftrightarrow& \left \{
		\begin{array}{rll}
			\omb([x,y],\varphi(\xib))&=0,&\forall x,y\in\hb. \qquad(^*)\\
			\omb([\xib,x],\varphi(\xib))&=0,&\forall x\in\hb.\qquad(^{**})
		\end{array}
		\right.
	\end{eqnarray*}
	The $(^*)$ condition is equivalent to $\omb(x,y*\varphi(\xib))=0$ for all $x,y\in\hb$ then $\overline{R}_{\varphi(\xib)}=0$.	
	
	Since $\omb$ is 2-cocycle $(^{**})$ becomes
	\[	\omb([\varphi(\xib),\xib],x)=0,\qquad\forall x\in\hb,\]
	as $\omb$ is non-degenerate 2-form on $\hb$, that holds $[\varphi(\xib),\xib]=0$. To complete the proof, it remains to verify that $\al\we\om^n \neq0$, which equivalent to prove that $\varPhi : \G\lr\G^*$ is an isomorphism. Indeed for $x\in\langle d \rangle\oplus\Gb\oplus\langle e \rangle$, we can write  $x=x_1d+\overline{x}+x_2\xib+x_3e$ with $\overline{x}\in\hb$ and $x_1,\,x_2,x_3\in\R$. A direct calculation gives us.
	\[\left\{
	\begin{array}{ll}
	\varPhi(x,\overline{y})&=\omb_{\hb}(\overline{x},\overline{y}),\qquad \forall \overline{y}\in\hb,\\
	\varPhi(x,\xib)&=-x_3,\\
	\varPhi(x,e)&=x_2,\\
	\varPhi(x,d)&=x_1.
	\end{array}
	\right.\] Which means that if $\varPhi(x,.)=0$,
	then $x=0$.
	
\end{proof}

\subsection{The  second construction}
Now let $(\Gb,\alb,\omb)$ be a cosymplectic Lie algebra and $\overline{[.,.]}$ its Lie bracket and let
\[\G=\langle d\rangle\oplus\Gb\oplus \langle e\rangle,\]
be a double extension  of $(\Gb,\overline{[\,,\,]})$ by $(D,\theta)$. We define a  2-form $\omega$ and  a 1-form $\alpha$ on the vector space $\G$ by requiring that
\[\omega=\overline{\omega}+d^*\we e^*,\]
and 
\[\alpha(x)=\overline{\alpha}(x),\;\;\forall x\in\Gb,\quad 
\alpha(d)\in \G,\quad\alpha(e)\in \G.\]
On the one hand , if we assume that $\alpha$ is a 1-cocycle, we get
\[
\left\{
\begin{array}{l}
\al([x,y])=\theta(x,y)\al(e)=0,\\
\al([d,x])=\alb(\varphi(x))+\lambda(x)\al(e)=0,\\
\al([d,e])=\al(v)+t\al(e)=0.
\end{array}
\right.\qquad (\dagger)
\]

On the other hand, while $\om$ is a 2-cocycle, consequently for all $x$, $y\in\Gb$
\[
\left\{
\begin{array}{ll}
\oint\om([x,y],d)&=\om(\overline{[x,y]},d)+\theta(x,y)\om(e,d)+\omb(\varphi(x),y)+\lambda(x)\om(e,y)-\omb(\varphi(y),x)-\lambda(y)\om(d,y)\\
&=-\theta(x,y)+\omb(\varphi(x),y)+\omb(x,\varphi(y))=0,\\
\oint\om([x,e],d)&=\omb(v,x)=0.
\end{array}
\right.\quad (\ddagger)
\]

We distinguish two cases.

\textbf{The first case}: $\al(e)=0$. We use relations $(\dagger)$ and $(\ddagger)$ to obtain
\begin{equation*}
\alb\circ\varphi=0\esp\al(v)=0,
\end{equation*}
\begin{equation*}
\theta=\omb_\varphi\esp v\in\ker(\omb).
\end{equation*}
Note that $\al(v)=0$ and $v\in\ker(\omb)$ trains that  $v=0$.

Note by $\om_{\varphi,\varphi}=\theta_\varphi$. By combining the conditions described above with Proposition $\ref{pr 3.1}$ we obtain the following result.
\begin{theo}
	Let $(\Gb,\alb,\omb)$ be a  cosymplectic Lie algebra with the Reeb vector $\xib$. Let $\G=\langle d\rangle\oplus\Gb\oplus \langle e\rangle$ be a  double extension   of $\Gb$ by  $(D,\omb_\varphi)$ with $D=(\varphi,\lambda,0,t)\in Der(\Gb)\times\Gb^*\times\Gb\times\R$. Then 
	\[{\om}= \omb+d^*\we e^*,\]
	\[\alpha(x)=\overline{\alpha}(x),\;\;\forall x\in\Gb,\quad 
	\alpha(d)\in \R,\]
	defines a cosymplectic structure on $\G$ if and only if 
	\begin{enumerate}
		\item  $\al(e)=0$,   \item $\alb\circ\varphi=0$,
		\item $t\om_\varphi-\om_{\varphi,\varphi}=\md\lambda$.
	\end{enumerate}
	Also, the Reeb vector is $\xib$.
\end{theo}
\begin{proof}
	It remains to verify that $\varPhi : \G\lr\G^*$ is an isomorphism. Let $x\in\langle d \rangle\oplus\Gb\oplus\langle e \rangle$ we can write  $x=x_1d+\overline{x}+x_2\xib+x_3e$, with $\overline{x}\in\hb$ and $x_1,\,x_2,\,x_3\in\R$. A direct calculation gives us 
	\[\left\{
	\begin{array}{ll}
	\varPhi(x,\overline{y})&=\omb_{\hb}(\overline{x},\overline{y}),\qquad \forall \overline{y}\in\hb,\\
	\varPhi(x,e)&=x_1,\\
	\varPhi(x,\xib)&=x_1\al(d)+x_2,\\
	\varPhi(x,d)&=-x_3+x_1\al^2(d)+x_2\al(d).
	\end{array}
	\right.\]
	Which means that if $\varPhi(x,.)=0$, then $x=0$.
\end{proof}
\textbf{The second case}: $\al(e)\not=0$. To simplify, we can always take $\al(e)=-1$.   We use relation $(\dagger)$ and $(\ddagger)$ to obtain
\begin{equation*}
\omb_\varphi=\theta=0\esp \varphi\in Der(\Gb),
\end{equation*}
\begin{equation*}
\lambda=\alb\circ\varphi\esp t=\alb(v),
\end{equation*}
\begin{equation*}
v\in\ker(\omb).
\end{equation*}
By combining the conditions described above  with Proposition $\ref{pr 3.1}$ we obtain the following result.
\begin{theo}
	Let $(\Gb,\alb,\omb)$ be a cosymplectic Lie algebra with the Reeb vector $\xib$. Let $\G=\langle d \rangle\oplus\Gb\oplus\langle e \rangle$ be a double extension   of $\Gb$, by $(D,\theta=0)$, with $D=(\varphi,\alb\circ\varphi,v,\alb(v))\in Der({\Gb})\times{\Gb}^*\times{\Gb}\times\R$. Then 
	\[{\om}= \omb+d^*\we e^*,\]
	\[\alpha(x)=\overline{\alpha}(x),\;\;\forall x\in\Gb,\quad 
	\alpha(d)\in \R,\quad\alpha(e)=-1,\]
	defines a cosymplectic structure in $\G$ if and only if
	\begin{enumerate}
		\item  $\omb_\varphi=0$,
		\item  $\alb\circ\varphi([x,y])=0$,\quad $\forall x,y\in\Gb$,
		\item $ v\in Z(\Gb)\cap \ker(\omb)$.
	\end{enumerate}
	Also, the Reeb vector is $\xib$.	
\end{theo}
\begin{proof}
	It remains to verify that $\varPhi : \G\lr\G^*$ is an isomorphism. Let $x\in\langle d \rangle\oplus\Gb\oplus\langle e \rangle$ we can write  $x=x_1d+\overline{x}+x_2\xib+x_3e$, with $\overline{x}\in\hb$ and $x_1,\,x_2,\,x_3\in\R$. A direct calculation gives us 
	\[\left\{
	\begin{array}{ll}
	\varPhi(x,\overline{y})&=\omb_{\hb},(\overline{x},\overline{y}),\qquad \forall \overline{y}\in\hb\\
	\varPhi(x,\xib)&=x_1\al(d)+x_2-x_3,\\
	\varPhi(x,e)&=x_1-x_1\al(d)-x_2+x_3,\\
	\varPhi(x,d)&=-x_3+x_1\al^2(d)+x_2\al(d)-x_3\al(d),
	\end{array}
	\right.\]
	which means that if $\varPhi(x,.)=0$, then $x=0$.
\end{proof}

\begin{exems} 
	In order to illustrate and compare these three constructions. We end this chapter with examples of five-dimensional cosymplectic Lie algebras construction from the same three-dimensional cosymplectic Lie algebra. 
	
	We consider the following three-dimensional cosymplectic Lie algebra:
	\[\Gb : [e_1,e_2]=e_1\esp\alb=e^3,\quad\omb=e^{12}.\]
	Let $\G=\langle e_4 \rangle\oplus\Gb\oplus\langle e_5 \rangle$ be a  double extension   of $(\Gb,\overline{[\,,\,]})$ by $(D,\theta)$, with $D=(\varphi,\lambda,v,t)$.
	
	It is straightforward to check that  derivations on the Lie algebra $\Gb$ take the following form
	\[\varphi=\begin{pmatrix}
	a&b&0\\0&0&0\\0&c&f
	\end{pmatrix},\quad a,b,c,f\in\R.\]
	Set $v=ze_3\in Z(\Gb)$ and $\la=\la_1e_1+\la_2e_2+\la_3e_3\in\Gb^*$  ,\quad$z,\la_i\in\R$.
	
	After standard calculations. 
	\begin{enumerate}
		\item[]\textit{The  first construction}, gives us the following cosymplectic Lie algebra
		\[\G_1 : \begin{array}{ll}
		[e_1,e_2]=e_1, &[e_4,e_3]=fe_3+\la_3e_5,\\
		{[e_4,e_2]=be_1}, &[e_4,e_5]=ze_3-fe_5.
		\end{array}\]
		\[\al=e^4,\quad\om=e^{12}+e^{35},\quad\xi=e_4.\]
		\item[]\textit{The  second construction}:
		\[\G_2 : \begin{array}{ll}
		[e_1,e_2]=e_1+ae_5, &[e_4,e_2]=be_1+\la_2e_5,\\
		{[e_4,e_1]=ae_1+(a^2-ta)e_5}, &[e_4,e_3]=\la_3e_5,\\
		{[e_4,e_5]=te_5.}
		\end{array}\]
		\[\al=e^3+xe^4,\quad\om=e^{12}+e^{45},\quad\xi=e_3,\quad x\in\R.\]
		
		\item[]\textit{The  third construction}:
		\[\G_3 : \begin{array}{ll}
		[e_1,e_2]=e_1, &[e_4,e_3]=f(e_3+e_5),\\
		{[e_4,e_2]=be_1+c(e_3+e_5)}, &[e_4,e_5]=t(e_3+e_5).
		\end{array}\]
		\[\al=e^3+xe^4-e^5,\quad\om=e^{12}+e^{45},\quad\xi=e_3,\quad x\in\R.\]
		
	\end{enumerate}
	Let us denote by $D(\G_i)$ the derived Lie algebra of $\G_i $, $i= 1,\ldots,3$. 
	It is clear that for a suitable choice of the constants $(b,f,\la_3,z)$ we have $\dim(D(\G_1))=3$, while $\dim(D(\G_2)) $ and $\dim(D(\G_3))$ are less than or equal to two. Therefore the cosymplectic Lie algebra $\G_1$ never comes from the other two constructions.
\end{exems}

\section{Classification in low dimensional}

\subsection{Three-dimensional cosymplectic Lie algebras}

Let $\G$ be a three-dimensional Lie algebra. Denoting by $\al=a_1e^1+a_2e^2+a_3e^3$ a 1-form and $\om=a_{12}e^{12}+a_{13}e^{13}+a_{23}e^{23}$ a  2-form on $\G$. On each three-dimensional Lie algebra we compute the 1 and 2-cocycle conditions for $\al$ and $\om$ respectively. Next, we calculate the rank of $\varPhi$, if $\varPhi$ has a maximum rank, then $\G$ supports a cosymplectic structure.
\begin{pr}
	Let $(\G,\al,\om)$ be a three-dimensional   cosymplectic real Lie algebra. Then
	$\G$ is isomorphic to one of the following Lie algebras equipped with a cosymplectic structure:
	\begin{enumerate}	
		\item[$\G_{2.1}\oplus \G_1:$] $[e_1,e_2]=e_1$  (decomposable solvable, Bianchi III)
		\begin{align*}
		\al&=a_2e^2+a_3e^3\\
		\om&=a_{12}e^{12}+a_{23}e^{23},	\qquad a_3a_{12}\not=0.	
		\end{align*}	
		\item[$\G_{3.1}:$] $[e_2,e_3]=e_1$  (Weyl algebra, nilpotent, Bianchi II)
		\begin{align*}
		\al&=a_2e^2+a_3e^3\\
		\om&=a_{12}e^{12}+a_{13}e^{13}+a_{23}e^{23},	\qquad a_3a_{12}-a_2a_{13}\not=0.	
		\end{align*}
		\item[$\G_{3.4}^{-1}:$]  $[e_1,e_3]=e_1$ and $[e_2,e_3]=-e_2$   (solvable, Bianchi VI, $a=-1$)
		\begin{align*}
		\al&=a_3e^3\\
		\om&=a_{12}e^{12}+a_{13}e^{13}+a_{23}e^{23},	\qquad a_3a_{12}\not=0.	
		\end{align*}
		\item[$\G_{3.5}^0:$]  $[e_1,e_3]=-e_2$ and $[e_2,e_3]=e_1$   (solvable, Bianchi VII, $\be=0$)
		\begin{align*}
		\al&=a_3e^3\\
		\om&=a_{12}e^{12}+a_{13}e^{13}+a_{23}e^{23},	\qquad a_3a_{12}\not=0.	
		\end{align*}
	\end{enumerate}
\end{pr}

Two  cosymplectic Lie algebra  $(\G_1,\al_1, \om_1)$ and $(\G_2,\al_2, \om_2)$ are said to be isomorphic if there exists a Lie algebra isomorphism $L : \G_1\lr\G_2$ such that
\[L^*(\al_2)=\al_1\esp L^*(\om_2)=\om_1.\]
The following proposition gives a complete classification  up to an automorphism to  cosymplectic structures on  three-dimensional Lie algebras.
\begin{pr}\label{pr 4.2}
	Let $(\G,\al,\om)$ be a three-dimensional   cosymplectic real Lie algebra. Then
	$(\G,\al,\om)$ is isomorphic to one of the following cosymplectic Lie algebras:
	\begin{enumerate}	
		\item[$\bullet$]$\G_{2.1}\oplus \G_1:$ $(\al,\om)=( e^3,e^{12})$.
		\item[]\qquad\qquad $(\al,\om)=(\la e^3,e^{12}+e^{23})$,\quad $\la\in\R-\{0\}$.
		\item[$\bullet$]$\G_{3.1}:$  $(\al,\om)=(\la e^2,e^{13})$,\quad$\la\in\R-\{0\}$.
		\item[$\bullet$]$\G_{3.4}^{-1}:$ $(\al,\om)=(\la e^3,e^{12})$,\quad $\la>0$.
		\item[$\bullet$]$\G_{3.5}^0:$ $(\al,\om)=(\la e^3,e^{12})$,\quad $\la>0$.
	\end{enumerate}
\end{pr}

\begin{proof}
	We proceed as follows. We act the automorphisms of $\G$ on $\om$ to find the simplest possible form $\om_0$ then we seek all the automorphisms which transforms $\om$ into $\om_0$ and finally we act these automorphisms on $\al$ in order to simplify it. These calculations are made using computation software Maple 18$^\circledR$.
	
	We will give the proof for  Lie algebra $\G_{4.3}^{-1}$: $[e_1,e_3]=e_1$, $[e_2,e_3]=-e_3$, since all cases must be treated in the same way. Cosymplectic structures in $\G_{4.3}^{-1}$ are given by the following family 
	\begin{align*}
	\al&=a_3e^3\\
	\om&=a_{12}e^{12}+a_{13}e^{13}+a_{23}e^{23},	\qquad a_3a_{12}\not=0.	
	\end{align*}
	In this case the automorphisms  are given  by
	\[T_1=\begin{pmatrix}
	t_{11}&0&t_{13}\\
	0&t_{22}&t_{23}\\ 
	0&0&1\end{pmatrix}\esp T_2=\begin{pmatrix}
	0&t_{12}&t_{13}\\
	t_{21}&0&t_{23}\\ 
	0&0&-1\end{pmatrix}.
	\] 
	
	The  automorphism $L=\begin{pmatrix}
	1&0&\frac {a_{23}}{a_{12}}\\ 0&\frac{1}{a_{12}}&-\frac{a_{13}}{a_{12}}\\ 
	0&0&1
	\end{pmatrix}$,  satisfying that  $L^*(\om)=e^{12}$,
	all automorphisms that satisfy $L^*(\om)=e^{12}$ are
	\[L_1=\begin{pmatrix} t_{{11}}&0&{\frac {a_{{23}}}{a_{{12}}}}\\ 
	0&{\frac{1}{t_{{11}}a_{{12}}}}&-{\frac{a_{{13}}}{a_{{12}}}}\\ 0&0&1\end{pmatrix}\esp L_2=\begin{pmatrix}
	0&t_{{12}}&-{\frac {a_{{23}}}{a_{{12}}}
	}\\-{\frac {1}{t_{{12}}a_{{12}}}}&0&{\frac {a_{13}}{a_{12}}}\\ 0&0&-1
	\end{pmatrix}.\]
	
	A direct calculation gives us $L_1^*(\al)=\al$ and $L_2^*(\al)=-\al$.  So, we can take $\al=\la e^3$ with $\la>0$.
\end{proof}

By  using Proposition \ref{pr2.1}, a direct calculation gives us the following corollary.
\begin{co}
	The left symmetric product associated to the three-dimensional cosymplectic Lie algebras is given by
	\begin{enumerate}	
		\item[$\bullet$]$\G_{2.1}\oplus \G_1:$  $e_1.e_2=e_1$,\, $e_2.e_2=e_2$.  
		\item[$\bullet$] $\G_{3.1}:$ $e_2.e_2=\frac{1}{\la^2}e_3$,\, $e_3.e_2=-e_1$.
		\item[$\bullet$]  $\G_{3.4}^{-1}:$ $e_1.e_2=e_2.e_1=\frac{1}{\la^2} e_3$,\; $e_3.e_1=-e_1$,\,  $e_3.e_2=e_2$.
		\item[$\bullet$] $\G_{3.5}^0:$ $e_1.e_1=\frac{1}{\la^2}e_3$,\, $e_2.e_2=\frac{1}{\la^2}e_3$,\, $e_3.e_1=e_2$,\, $e_3.e_2=-e_1$.
	\end{enumerate}
\end{co}

\subsection{Five-dimensional cosymplectic Lie algebras }

In \cite{C-P} the authors give a classification of five-dimensional cosymplectic Lie algebras, starting from four-dimensional symplectic Lie algebras endowed with a derivation. The disadvantage of this method is that the correspondence with five-dimensional Lie algebras well known in the literature is ignored. We propose here a direct method (case by case), by searching among the $40$ five-dimensional Lie algebras (listed in \cite{P-W}), those which are cosymplectic and we give their structures.

\begin{pr}
	Let $(\G,\al,\om)$ be a five-dimensional  cosymplectic real Lie algebra. Then
	$(\G,\al,\om)$ is isomorphic to one of the following cosymplectic Lie algebras:
	
	\begin{enumerate}	
		\item[$A_{5,1}:$] $[e_3,e_5] = e_1$, $[e_4,e_5]=e_2$  (nilpotent)
		\begin{align*}
		\al&=a_3e^3+a_4e^4+a_5e^5\\
		\om&=a_{13}e^{13}+a_{23}e^{14}+a_{15}e^{15}+a_{23}e^{23}+a_{24}e^{24}+a_{25}e^{25}+a_{34}e^{34}+a_{35}e^{35}+a_{45}e^{45}\\
		&with\qquad a_3a_{15}a_{24}-a_3a_{23}a_{25}+a_4a_{13}a_
		{25}-a_4a_{15}a_{23}-a_5a_{13}a_{24}+a_5{a_
			{23}^{2}}\not=0.	
		\end{align*}
		\item[$A_{5,2}:$] $[e_2,e_5] = e_1$, $[e_3,e_5]=e_2$, $[e_4,e_5] = e_3$ 	(nilpotent)
		\begin{align*}
		\al&=a_4e^4+a_5e^5\\
		\om&=-a_{23}e^{14}+a_{15}e^{15}+a_{23}e^{23}+a_{25}e^{25}+a_{34}e^{34}+a_{35}e^{35}+a_{45}e^{45}\\	
		&with\qquad \qquad a_{23}(a_4a_{15}+a_5a_{23})\not=0.
		\end{align*}
		\item[$A_{5,5}:$] $[e_2,e_5] = e_1$, $[e_3,e_4]=e_1$, $[e_3,e_5] = e_2$ 	(nilpotent)	
		\begin{align*}
		\al&=a_3e^3+a_4e^4+a_5e^5\\
		\om&=a_{24}e^{15}+a_{23}e^{23}+a_{24}e^{24}+a_{25}e^{25}+a_{34}e^{34}+a_{35}e^{35}+a_{45}e^{45}\\
		&with\qquad a_{24}(a_3a_{24}-a_4a_{23}) \not=0.	
		\end{align*}		
		
		\item[$A_{5,6}:$] $[e_2,e_5] = e_1$, $[e_3,e_5]=e_2$, $[e_3,e_4]=e_1$, $[e_4,e_5] = e_3$ 	(nilpotent)
		\begin{align*}
		\al&=a_4e^4+a_5e^5\\
		\om&=-a_{23}e^{14}+a_{24}e^{15}+a_{23}e^{23}+a_{24}e^{24}+a_{25}e^{25}+a_{34}e^{34}+a_{35}e^{35}+a_{45}e^{45}\\	&with \qquad a_{23}(a_4a_{24}+a_5a_{23})\not=0.
		\end{align*}
		\item[$A^{a,-a,-1}_{5,7}:$] $[e_1, e_5] = e_1$, $[e_2,e_5]=ae_2$, $[e_3, e_5] =-ae_3$, $[e_4, e_5]=-e_4$,\quad $a\in \left] -1,0 \right[  \cup \left]0,1 \right[$	.
		\begin{align*}
		\al&=a_5e^5\\
		\om&=a_{14}e^{14}+a_{15}e^{15}+a_{23}e^{23}+a_{24}e^{24}+a_{25}e^{25}+a_{34}e^{34}+a_{35}e^{35}+a_{45}e^{45}\\	&with \qquad a_5a_{14}a_{23}\not=0.
		\end{align*}
		
		\item[$A^{1,-1,-1}_{5,7}:$] $[e_1, e_5] = e_1$, $[e_2,e_5]=e_2$,  $[e_3, e_5] =-e_3$, $[e_4, e_5]=-e_4$	
		\begin{align*}
		\al&=a_5e^5\\
		\om&=a_{13}e^{13}+a_{14}e^{14}+a_{15}e^{15}+a_{23}e^{23}+a_{24}e^{24}+a_{25}e^{25}+a_{35}e^{35}+a_{45}e^{45}\\	&with \qquad a_5(a_{24}a_{13}-a_{14}a_{2 3})\not=0.
		\end{align*}
		\item[$A^{-1}_{5,8}:$]$[e_2, e_5] = e_1$,$[e_3, e_5] =e_3$, $[e_4, e_5]=-e_4$	 	
		\begin{align*}
		\al&=a_2e^2+a_5e^5\\
		\om&=a_{12}e^{12}+a_{15}e^{15}+a_{25}e^{25}+a_{34}e^{34}+a_{35}e^{35}+a_{45}e^{45}\\	&with \qquad a_{34}(a_5a_{12}-a_2a_{15}) \not=0.
		\end{align*}
		\item[$A^{0,-1}_{5,9}:$]$[e_1, e_5] = e_1$, $[e_2,e_5]=e_1+e_2$, $[e_4, e_5] =-e_4$	 	
		\begin{align*}
		\al&=a_3e^3+a_5e^5\\
		\om&=a_{15}e^{15}+a_{24}e^{24}+a_{25}e^{25}+a_{35}e^{35}+a_{45}e^{45}\\	&with \qquad a_3a_{15}a_{24} \not=0.
		\end{align*}
		\item[$A^{-1,0,q}_{5,13}:$]$[e_1, e_5] = e_1$, $[e_2,e_5]=-e_2$, $[e_3, e_5] =-qe_4$, $[e_4, e_5]=qe_3$	 	
		\begin{align*}
		\al&=a_5e^5\\
		\om&=a_{12}e^{12}+a_{15}e^{15}+a_{25}e^{25}+a_{34}e^{34}+a_{35}e^{35}+a_{45}e^{45}\\	&with \qquad a_5a_{12}a_{34} \not=0.
		\end{align*}
		
		\item[$A^{0}_{5,14}:$]$[e_2, e_5] = e_1$, $[e_3, e_5] =-e_4$, $[e_4, e_5]=e_3$	 	
		\begin{align*}
		\al&=a_2e^2+a_5e^5\\
		\om&=a_{12}e^{12}+a_{15}e^{15}+a_{25}e^{25}+a_{34}e^{34}+a_{35}e^{35}+a_{45}e^{45}\\	&with \qquad a_{34}(a_5a_{12}-a_2a_{15}) \not=0.
		\end{align*}
		\item[$A^{-1}_{5,15}:$]$[e_1, e_5] = e_1,[e_2, e_5] = e_1+e_2$, $[e_3, e_5] =-e_3$, $[e_4, e_5]=e_3-e_4$	 	
		\begin{align*}
		\al&=a_5e^5\\
		\om&=-a_{23}e^{14}+a_{15}e^{15}+a_{23}e^{23}+a_{24}e^{24}+a_{25}e^{25}+a_{35}e^{35}+a_{45}e^{45}\\	&with \qquad a_5a_{23} \not=0.
		\end{align*}
		\item[$A^{1,p,-p}_{5,17}:$] $[e_1, e_5] = pe_1-e_2$, $[e_2,e_5]=e_1+pe_2$, $[e_3, e_5] =-pe_3-e_4$, $[e_4, e_5]=e_3-pe_4$,$\quad p \not=0$.		
		\begin{align*}
		\al&=a_5e^5\\
		\om&=a_{24}e^{13}-a_{23}e^{14}+a_{15}e^{15}+a_{23}e^{23}+a_{24}e^{24}+a_{25}e^{25}+a_{35}e^{35}+a_{45}e^{45}\\	&with \qquad a_5(a^{2}_{23}+a^{2}_{2 4})\not=0.
		\end{align*}
		\item[$A^{1,0,0}_{5,17}:$] $[e_1, e_5] = -e_2$, $[e_2,e_5]=e_1$, $[e_3, e_5] =-e_4$, $[e_4, e_5]=e_3$	
		\begin{align*}
		\al&=a_5e^5\\
		\om&=a_{12}e^{12}+a_{24}e^{13}-a_{23}e^{14}+a_{15}e^{15}+a_{23}e^{23}+a_{24}e^{24}+a_{25}e^{25}+a_{34}e^{34}+a_{35}e^{35}+a_{45}e^{45}\\	&with \qquad a_5(a_{12}a_{34}-a^{2}_{23}-a^{2}_{24})\not=0.
		\end{align*}
		\item[$A^{-1,p,-p}_{5,17}:$] $[e_1, e_5] = pe_1-e_2$, $[e_2,e_5]=e_1+pe_2$, $[e_3, e_5] =-pe_3+e_4$, $[e_4, e_5]=-e_3-pe_4$,$\quad p \not=0$.		
		\begin{align*}
		\al&=a_5e^5\\
		\om&=-a_{24}e^{13}+a_{23}e^{14}+a_{15}e^{15}+a_{23}e^{23}+a_{24}e^{24}+a_{25}e^{25}+a_{35}e^{35}+a_{45}e^{45}\\	&with \qquad a_5(a^{2}_{23}+a^{2}_{2 4})\not=0.
		\end{align*}	
		\item[$A^{-1,0,0}_{5,17}:$] $[e_1, e_5] = -e_2$, $[e_2,e_5]=e_1$, $[e_3, e_5] =e_4$, $[e_4, e_5]=-e_3$	
		\begin{align*}
		\al&=a_5e^5\\
		\om&=a_{12}e^{12}-a_{24}e^{13}+a_{23}e^{14}+a_{15}e^{15}+a_{23}e^{23}+a_{24}e^{24}+a_{25}e^{25}+a_{34}e^{34}+a_{35}e^{35}+a_{45}e^{45}\\	&with \qquad a_5(a_{12}a_{34}+a^{2}_{23}+a^{2}_{24})\not=0.
		\end{align*}
		\item[$A^{s,0,0}_{5,17}:$] $[e_1, e_5] = -e_2$, $[e_2,e_5]=e_1$, $[e_3, e_5] =-se_4$, $[e_4, e_5]=se_3$,$\qquad s\not\in\{-1,0,1\}$.	
		\begin{align*}
		\al&=a_5e^5\\
		\om&=a_{12}e^{12}+a_{15}e^{15}+a_{25}e^{25}+a_{34}e^{34}+a_{35}e^{35}+a_{45}e^{45}\\	&with \qquad a_5a_{12}a_{34}\not=0.
		\end{align*}
		\item[$A^{0}_{5,18}:$] $[e_1, e_5] = -e_2$, $[e_2,e_5]=e_1$, $[e_3, e_5] =e_1-e_4$, $[e_4, e_5]=e_2+e_3$	
		\begin{align*}
		\al&=a_5e^5\\
		\om&=a_{24}e^{13}+a_{15}e^{15}+a_{24}e^{24}+a_{25}e^{25}+a_{34}e^{34}+a_{35}e^{35}+a_{45}e^{45}\\	&with \qquad a_5a_{24}\not=0.
		\end{align*}
		\item[$A^{1,-1}_{5,19}:$] $[e_1, e_5] = e_1$, $[e_2, e_3] = e_1$, $[e_2,e_5]=e_2$,  $[e_4, e_5]=-e_4$	
		\begin{align*}
		\al&=a_3e^3+a_5e^5\\
		\om&=a_{23}e^{15}+a_{23}e^{23}+a_{24}e^{24}+a_{25}e^{25}+a_{35}e^{35}+a_{45}e^{45}\\	&with \qquad a_3a_{23}a_{24}\not=0.
		\end{align*}	
		\item[$A^{\frac{1}{2},-1}_{5,19}:$] $[e_1, e_5] = \frac{1}{2}e_1$, $[e_2, e_3] = e_1$, $[e_2,e_5]=e_2$, $[e_3, e_5] =-\frac{1}{2}e_3$, $[e_4, e_5]=-e_4$	
		\begin{align*}
		\al&=a_5e^5\\
		\om&=a_{13}e^{13}+a_{15}e^{15}+2a_{15}e^{23}+a_{24}e^{24}+a_{25}e^{25}+a_{35}e^{35}+a_{45}e^{45}\\	&with \qquad a_5a_{13}a_{24}\not=0.
		\end{align*}
		\item[$A^{-1,2}_{5,19}:$] $[e_1, e_5] = -e_1$, $[e_2, e_3] = e_1$, $[e_2,e_5]=e_2$, $[e_3, e_5] =-2e_3$, $[e_4, e_5]=2e_4$	
		\begin{align*}
		\al&=a_5e^5\\
		\om&=a_{12}e^{12}-a_{23}e^{15}+a_{23}e^{23}+a_{25}e^{25}+a_{34}e^{34}+a_{35}e^{35}+a_{45}e^{45}\\	&with \qquad a_5a_{12}a_{34}\not=0.
		\end{align*}
		\item[$A^1_{5,30}:$]$[e_2, e_4] = e_1$, $[e_3, e_4] = e_2$, $[e_1, e_5] = 2e_1$, $[e_2,e_5]=e_2$, $[e_4, e_5]=e_4$ 	
		\begin{align*}
		\al&=a_3e^3+a_5e^5\\
		\om&=2a_{24}e^{15}+a_{24}e^{24}+a_{34}e^{25}+a_{34}e^{34}+a_{35}e^{35}+a_{45}e^{45}\\	&with \qquad a_3a_{24}\not=0.
		\end{align*}
		\item[$A^{0,-1}_{5,33}:$] $[e_1, e_4] = e_1$, $[e_2,e_5]=e_2$, $[e_3, e_4] =-e_3$
		\begin{align*}
		\al&=a_4e^4+a_5e^5\\
		\om&=a_{13}e^{13}+a_{14}e^{14}+a_{25}e^{25}+a_{34}e^{34}+a_{45}e^{45}\\	&with \qquad a_4a_{13}a_{25}\not=0.
		\end{align*}
		\item[$A^{-1,0}_{5,33}:$] $[e_1, e_4] = e_1$, $[e_2,e_5]=e_2$, $[e_3, e_5] =-e_3$
		\begin{align*}
		\al&=a_4e^4+a_5e^5\\
		\om&=a_{14}e^{14}+a_{23}e^{23}+a_{25}e^{25}+a_{35}e^{35}+a_{45}e^{45}\\	&with \qquad a_5a_{14}a_{23}\not=0.
		\end{align*}
		\item[$A_{5,36}:$] $[e_1,e_4]=e_1$, $[e_2, e_3] = e_1$,  $[e_2, e_4] =e_2$, $[e_2, e_5] =-e_2$, $[e_3, e_5] =e_3$
		\begin{align*}
		\al&=a_4e^4+a_5e^5\\
		\om&=a_{23}e^{14}+a_{23}e^{23}-a_{25}e^{24}+a_{25}e^{25}+a_{35}e^{35}+a_{45}e^{45}\\	&with \qquad a_5a_{23}\not=0.
		\end{align*}
		\item[$A_{5,37}:$] $[e_1,e_4]=2e_1$, $[e_2, e_3] = e_1$,  $[e_2, e_4] =e_2$, $[e_3, e_4] =e_3$, $[e_2, e_5] =-e_3$, $[e_3, e_5] =e_2$
		\begin{align*}
		\al&=a_4e^4+a_5e^5\\
		\om&=2a_{23}e^{14}+a_{23}e^{23}+a_{35}e^{24}-a_{34}e^{25}+a_{34}e^{34}+a_{35}e^{35}+a_{45}e^{45}\\	&with \qquad a_5a_{23}\not=0.
		\end{align*}			
	\end{enumerate}
	
\end{pr}

Unlike the three-dimensional Heisenberg algebra, the five-dimensional Heisenberg algebra does not support any cosymplectic structure. Therefore we have the following most general result.
\begin{pr}
	Let $\mathrm{H}_{2n+1}$ be the ($2n+1$)-dimensional Heisenberg Lie algebra   generated by elements $\{e_i,f_i,z\}_{\{1\leq i\leq n\}}$,  with the relations: $[e_i,f_i]=z$. Then  $\mathrm{H}_{2n+1}$ supports a cosymplectic structure if and only if $n=1$.
\end{pr}
\begin{proof} Proposition \ref{pr 4.2} shows that $\mathrm{H}_3=\G_{3.1}$ admits cosymplectic structures. Let $n \geq2$, and $\al\in\mathrm{H}^*_{2n+1}$, $\om\in\wedge^2\mathrm{H}^*_{2n+1}$  be a $1$-cocycle  and $2$-cocycle respectively. On the one hand by using the Maurer-Cartan equations of $\mathrm{H}_{2n+1}$ we get $\al(z)=0$, on the other hand for all $e_i$, $e_j$, $f_j$,\, $j\not=i$  we have  
	\begin{align*}
	\om(e_i,z)&=\om(e_i,[e_j,f_j])\\
	&=	\oint\om(e_i,[e_j,f_j])=0,
	\end{align*}
	in the same way, we find $\om(f_i,z)=0$ for all $1\leq i\leq n$. It follows that
	$\varPhi(e_i,z)=\varPhi(f_i,z)=0$ for all $1\leq i\leq n$. 
\end{proof}

Set by $\aff(2,\R)$ the affine  Lie algebra,  generated by elements $\{e_1,...,e_6\}$,  with the relations:
\[\begin{array}{lllll}
[e_1,e_3]=-e_1,&[e_2,e_4]=-e_1,&[e_3,e_4]=e_4,&[e_4,e_5]=e_3-e_6,&[e_5,e_6]=-e_5,\\
{[e_1,e_5]=-e_2},&[e_2,e_6]=-e_2,&[e_3,e_5]=-e_5,&[e_4,e_6]=e_4.&
\end{array}\]
\begin{pr}
	The only non-solvable cosymlectic Lie algebra of dimension less than or equal to seven is isomorphim to
	\[(\aff(2,\R)\ltimes\langle e_7 \rangle,\al,\om),\]
	the Lie brackets are given by that of $\aff(2,\R)$, to which we add this new entry: $[e_7,e_6]=\lambda e_2$, $[e_7,e_4]= \lambda e_1$, $\lambda\in\R$ and the cosymlectic structure is $\al=e^7$, $\om=e^{15}+e^{26}+e^{34}+e^{46}$. 
\end{pr}

\begin{proof}
	Let $(\G,\al,\om)$ be a five-dimensional cosymplectic non-solvable Lie algebra, it is well known that (see \cite{M-R}) any four-dimensional symplectic Lie algebra is solvable, by  Proposition \ref{pr 3.3}, we deduce that $\G$ contains an ideal of codimension $1$, this contradicts the fact that $\G$ is non-solvable.

	Now let $(\G,\al,\om)$ be a seven-dimensional cosymplectic non-solvable Lie algebra, by  Proposition \ref{pr 3.3}, $(\h,\om_\h)$ is a six-dimensional non-solvable symplectic Lie algebra. On the one hand, it is known that (see \cite{B-M}) the only non-solvable six-dimensional Lie algebra is (up to isomorphism) $(\aff(2,\R),\om_0)$ with		
	\[\om_0=e^{15}+e^{26}+e^{34}+e^{46}.\]		
	On the other  hand, in $\aff(2,\R)$  any derivation $D$  is inner (see \cite{D-M}), then there exist $x\in\h$ such that  $D=\ad_x$. The $2$-cocycle condition  of $\om_0$ gives that  
	\[\om_0(\ad_xy,z)+\om_0(y,\ad_xz)=\om_0(x,[y,z]),\]
	since $D$ is an i.s.t, it follows that, $\om_0(x,[y,z])=0$, $\forall y,z\in\h$,  by direct  computation we find that $x=\lambda e_2$, $\lambda\in\R$, by Proposition \ref{pr 3.3} we completes the proof.
\end{proof}	
\begin{remark}
	A five-dimensional cosymplectic Lie algebra is necessarily solvable.
\end{remark}

\end{document}